\documentclass[english,letterpaper,11pt,reqno]{amsart}
\usepackage[backend=bibtex,style=alphabetic,maxnames=6]{biblatex}
\usepackage{appendix}
\usepackage{amsbsy}
\usepackage{scalerel}
\usepackage{tcolorbox}
\usepackage{amsfonts}
\usepackage{amsmath}
\usepackage{amssymb}
\usepackage{amsthm}
\usepackage{graphicx}
\usepackage{ifthen}
\usepackage{textcomp}
\usepackage{enumitem, color, amssymb}
\usepackage{dsfont}
\usepackage{mathrsfs}
\usepackage{calrsfs}
\usepackage[bookmarksnumbered,colorlinks]{hyperref}
\hypersetup{colorlinks=true, linkcolor=blue, citecolor=blue}
\usepackage{cancel}
\usepackage{setspace}

\newcommand{\lra}{\longrightarrow}

\newcommand{\RR}{\mathbb{R}}

\makeatletter
\newcommand*{\defeq}{\mathrel{\rlap{%
                     \raisebox{0.25ex}{$\m@th\cdot$}}%
                     \raisebox{-0.25ex}{$\m@th\cdot$}}%
                     =}
\makeatother

\makeatletter
\newcommand*\owedge{\mathpalette\@owedge\relax}
\newcommand*\@owedge[1]{%
  \mathbin{%
    \ooalign{%
      $#1\m@th\bigcirc$\cr
      \hidewidth$#1\m@th\wedge$\hidewidth\cr
    }%
  }%
}
\makeatother
\newcommand{\kn}{\ensuremath{\raisebox{.04 em}{\,${\scriptstyle \owedge}$\,}}}  

\newtheorem{thm}{Theorem}

\newtheorem{cor}{Corollary}

\newtheorem{prop}{Proposition}
\newtheorem*{definition-non}{Definition}
\newtheorem*{theorem-non}{Theorem}
\newtheorem*{proposition-non}{Proposition}
\newtheorem*{lemma-non}{Lemma}
\newtheorem*{corollary-non}{Corollary}

\newcommand{\beqa}{\begin{eqnarray}}
\newcommand{\beq}{\begin{equation}}
\newcommand{\eeqa}{\end{eqnarray}}
\newcommand{\eeq}{\end{equation}}

\newcommand\ipr[2]{\langle {#1},{#2}\rangle}

\newcommand\ww[2]{#1 \wedge #2}

\newcommand\imp{\hspace{.2in}\Rightarrow\hspace{.2in}}

\newcommand\Hess{\text{Hess}}
\newcommand\what{\hat}

\newcommand\Ric{\text{Ric}}

\newcommand\Rmr{\text{Rm}_{\scalebox{0.6}{$g$}}}

\newcommand\comma{\hspace{.2in},\hspace{.2in}}
\newcommand\commas{\hspace{.1in},\hspace{.1in}}

\newcommand\hsr{*}

\newcommand\co{\hat{R}}

\newcommand\Rsec{\text{sec}_{\scalebox{0.6}{$g$}}}

\newtheorem{innercustomgeneric}{\customgenericname} 
\providecommand{\customgenericname}{}
\newcommand{\newcustomtheorem}[2]{%
  \newenvironment{#1}[1]
  {%
   \renewcommand\customgenericname{#2}%
   \renewcommand\theinnercustomgeneric{##1}%
   \innercustomgeneric
  }
  {\endinnercustomgeneric}
}

\makeatletter
\def\Ddots{\mathinner{\mkern1mu\raise\p@
\vbox{\kern7\p@\hbox{.}}\mkern2mu
\raise4\p@\hbox{.}\mkern2mu\raise7\p@\hbox{.}\mkern1mu}}
\makeatother

\newcustomtheorem{customthm}{Theorem}
\newcustomtheorem{customlemma}{Lemma}

\begin{document}
\title[]{On normal forms of gradient Ricci 4-solitons}
\author[]{Amir Babak Aazami}
\address{Clark University\hfill\break\indent
Worcester, MA 01610}
\email{aaazami@clarku.edu}

\maketitle
\begin{abstract}
In this note we analyze the normal form of the operator \mbox{$\hat{R} + \frac{1}{2}\hat{H}$} of a gradient Ricci 4-soliton in \cite{cao}. In particular, we show that the curvature operator $\co$ of the Koiso--Cao soliton inherits this normal form. By \cite{johnson}, this yields a normal form for the curvature operator of the Koiso--Cao soliton relative to the space of algebraic K\"ahler curvature operators.
\end{abstract}

\section{Introduction}
The concept of a \emph{normal form} for a Riemannian manifold \cite{berger,thorpe2} is motivated by a well known fact from linear algebra, namely, that on an inner product space $(V,\ipr{\,}{})$, a self-adjoint linear map $T\colon V \lra V$ is completely determined by the critical points and values of its quadratic form $v \mapsto \ipr{Tv}{v}$ restricted to the unit sphere. The relevance of this fact for Riemannian geometry is that the quadratic form of the curvature operator $\co\colon \Lambda^2 \lra \Lambda^2$ of a Riemannian manifold $(M,g)$ is (up to sign) precisely the sectional curvature: For any orthonormal pair $x,y \in T_pM$,
$$
\ww{x}{y} \mapsto \ipr{\co(\ww{x}{y})}{\ww{x}{y}} = \underbrace{\,\Rmr(x,y,x,y)\,}_{\text{$-\Rsec(\ww{x}{y})$}}.
$$
Because the sectional curvature $\Rsec$ determines the curvature tensor $\Rmr$, a natural question arises: When is $\co$ determined by \emph{just} the critical points and values of $\Rsec$ (relative to some class of algebraic curvature operators)? Here the latter is to be thought of as a smooth function \mbox{$\Rsec\colon \text{Gr}_2(T_pM)\lra \RR$} on the Grassmannian of 2-planes at $T_pM$. If so, then we say that the curvature operator has a ``normal form" at $p \in M$; this is a pointwise condition.
\vskip 6pt
Space forms, locally conformally flat manifolds, any Riemannian submanifold of a space form with a flat normal bundle\,---\,all of these have normal forms. By \cite{johnson}, so do K\"ahler 4-manifolds with positive-definite Ricci tensor (relative to the space of algebraic K\"ahler curvature operators).  But the most famous instances of normal forms are the class of \emph{Einstein} curvature operators in dimension 4, because here the normal form is derived from the fact that $\co$ commutes with the Hodge star operator $*$:
$$
\Ric = \lambda g \iff \hsr\co = \co\hsr.
$$
Hence $\co,*$ have a common orthonormal basis of eigenvectors. This fact, together with the splitting of $\Lambda^2 = \Lambda^+ \oplus \Lambda^-$ into the self-dual $(\Lambda^+)$ and anti-self-dual $(\Lambda^-)$ eigenspaces of $*$, and the fact that a 2-plane $P$ is a critical point of $\Rsec$ if and only if $\co P \in \text{span}\{P,*P\}$, was used to derive an orthonormal basis $\{e_1,e_2,e_3,e_4\} \subseteq T_pM$ all of whose corresponding 2-planes $\ww{e_i}{e_j}$ are critical points of $\Rsec$, with respect to which
\beqa
\label{eqn:AB}
\co = \begin{bmatrix}A_0&B_0\\B_0&A_0\end{bmatrix} \commas A_0 \defeq \text{diag}(\lambda_1,\lambda_2,\lambda_3) \commas B_0 \defeq \text{diag}(\mu_1,\mu_2,\mu_3).
\eeqa
Here the $\lambda_i$ are critical values of $\Rsec$, and the $\mu_i$ are not critical values but can be derived from them (with $\mu_1+\mu_2+\mu_3=0$ by the algebraic Bianchi identity). Furthermore, two Einstein operators with the same critical points and values are necessarily pointwise isometric; hence, within the Einstein class, the operator is determined up to orthogonal change by this data.
\vskip 6pt
The purpose of this note is to analyze another very interesting and recent example of a normal form in dimension 4. Namely, \cite{cao} showed that a normal form also exists for \emph{gradient Ricci 4-solitons},
\beqa
\label{eqn:sol}
\Ric + \text{Hess}f = \lambda g,
\eeqa
although not for the curvature operator $\co$ itself, but rather for the modified operator
$\hat{S} \defeq \co + \frac{1}{2}\hat{H}\colon\Lambda^2\lra \Lambda^2$, where $H \defeq \text{Hess}f \kn g$ is given by
\beqa
\ipr{\hat{H}(\ww{e_i}{e_j})}{\ww{e_k}{e_l}} \!\!&\defeq&\!\! H(e_i,e_j,e_k,e_l)\nonumber\\
 &=&\!\! \text{Hess}f(e_i,e_l)\,g_{jk} + \text{Hess}f(e_j,e_k)\,g_{il}\nonumber\\
&&\hspace{.2in}-\,\text{Hess}f(e_i,e_k)\,g_{jl} - \text{Hess}f(e_j,e_l)\,g_{ik}.\nonumber
\eeqa
In \cite[Remark~2.6]{cao}, it was asked whether $\co$ inherits the normal form of $\hat{S}$. We show that in at least one instance, the answer is yes:
\begin{theorem-non}
The curvature operator of the Koiso--Cao soliton on $\mathbb{CP}^2\#(-\mathbb{CP}^2)$ inherits the normal form of $\hat{S}$. As a consequence, the curvature operator of the Koiso--Cao soliton has a normal form relative to the space of algebraic K\"ahler curvature operators.
\end{theorem-non}

The Koiso--Cao soliton \cite{koiso,cao1996} will be analyzed in Theorem \ref{thm:KC} and Corollary \ref{cor:Johnson} in Section \ref{sec:proof} below.

\section{Overview of the normal form of $\hat{S} = \co + \frac{1}{2}\hat{H}$}
First, let $\{e_1,e_2,e_3,e_4\} = \{\partial_1,\partial_2,\partial_3,\partial_4\}$ arise from normal coordinates centered at $p \in M$, so that $\text{Hess}f(e_i,e_j)\big|_p = f_{ij}\big|_p$ and hence
\beqa
\label{eqn:Hcoord}
H_{ijkl}\big|_p = f_{il}\delta_{jk} + f_{jk}\delta_{il}-f_{ik}\delta_{jl} - f_{jl}\delta_{ik}\,\big|_p.
\eeqa
What makes the normal form of $\hat{S}$ particularly interesting is that it is ``Einstein-like," namely, the normal form arises from the fact that $\hat{S}$ commutes with $*$,
$\hsr \hat{S} = \hat{S} \hsr,$
and so the normal form basis $\{\ww{e_i}{e_j}\} \subseteq \Lambda^2(T_pM)$ of critical 2-planes is derived just as it was for $\co$ in \cite{berger,thorpe2} in the Einstein case (once again, a 2-plane $P$ will be a critical point of this quadratic form if and only if $\hat{S}P \in \text{span}\{P,*P\}$, by the same proof as in \cite{thorpe2}).  Of course, the quadratic form here is not the sectional curvature $\Rsec$, but rather that of $\hat{S}$:
\beqa
\ww{e_i}{e_j} \mapsto \ipr{\hat{S}(\ww{e_i}{e_j})}{\ww{e_i}{e_j}} \!\!&=&\!\! \ipr{\co(\ww{e_i}{e_j})}{\ww{e_i}{e_j}} + \frac{1}{2}\ipr{\hat{H}(\ww{e_i}{e_j})}{\ww{e_i}{e_j}}\nonumber\\
&=&\!\! \Rsec(\ww{e_i}{e_j}) -f_{ik}\delta_{jl} - f_{jl}\delta_{ik}.\nonumber
\eeqa
We would like to investigate what $\hat{S}$'s normal form says about $\co$. To do so, let us use the \emph{Hodge basis} $\{\xi_1^+,\xi_2^+,\xi_3^+,\xi_1^-,\xi_2^-,\xi_3^-\}$, where
$$
\Big\{\underbrace{\,\frac{1}{\sqrt{2}}(\ww{e_1}{e_2}\pm\ww{e_3}{e_4})\,}_{\text{$\xi_1^{\pm}$}},\underbrace{\,\frac{1}{\sqrt{2}}(\ww{e_1}{e_3}\pm\ww{e_4}{e_2})\,}_{\text{$\xi_2^{\pm}$}},\underbrace{\,\frac{1}{\sqrt{2}}(\ww{e_1}{e_4}\pm\ww{e_2}{e_3})\,}_{\text{$\xi_3^{\pm}$}}\Big\}
$$
comprise the self-dual and anti-self-dual eigenvectors of $*$ \mbox{($*\xi_i^{\pm}=\pm \xi_i^{\pm}$)}.  With respect to this basis,
$$
\co = \begin{bmatrix}\hat{W}^++\frac{\text{scal}}{12}I & K\\K^t & \hat{W}^-+\frac{\text{scal}}{12}I\end{bmatrix} \comma \hat{H} = \begin{bmatrix}-\frac{\Delta f}{2}I & C\\C^t & -\frac{\Delta f}{2}I\end{bmatrix},
$$
where $I$ is the $3 \times 3$ identity matrix, $\text{scal}$ is the scalar curvature of $g$ at $p$, $\hat{W}^{\pm}$ are the self-dual and anti-self dual portions of the Weyl curvature operator $\hat{W}$ that leave invariant the $\pm1$-eigenspaces of $*$, and
$$
-\frac{1}{2}K = C \defeq \begin{bmatrix}\frac{-f_{11}-f_{22}+f_{33}+f_{44}}{2} & f_{14}-f_{23}& -f_{13}-f_{24}\\
-f_{14}-f_{23} & \frac{-f_{11}+f_{22}-f_{33}+f_{44}}{2} & f_{12}-f_{34}\\
f_{13}-f_{24} & -f_{12}-f_{34}& \frac{-f_{11}+f_{22}+f_{33}-f_{44}}{2}\end{bmatrix}\cdot
$$
(See \cite[Lemma~2.3]{cao}, though note that in the latter the opposite sign convention for $\kn$ is being used.) When $\hat{S} = \co + \frac{1}{2}\hat{H}$ is expressed with respect to this Hodge basis, we have
\beqa
\label{eqn:CTsimple}
\hat{S} = \begin{bmatrix}A+B & O\\O&A-B\end{bmatrix} \imp \co = \begin{bmatrix}A+B+\frac{\Delta f}{4}I & -\frac{1}{2}C\\-\frac{1}{2}C^t&A-B+\frac{\Delta f}{4}I\end{bmatrix},
\eeqa
where $A \defeq \text{diag}(a_1,a_2,a_3)$ and $B \defeq \text{diag}(b_1,b_2,b_3)$. (If $\hat{S}$ had been expressed with respect to $\{\ww{e_i}{e_j}\}$ directly, then $\hat{S} = \begin{bmatrix}A&B\\B&A\end{bmatrix}$, similarly to \eqref{eqn:AB}.) 
With these preliminaries established, we now explore some properties of this $\hat{S}$-normal form basis $\{\ww{e_i}{e_j}\}$ and its relation to $\co$.

\section{Properties of the normal form of $\hat{S} = \co + \frac{1}{2}\hat{H}$}
Unfortunately, the existence of $\{\ww{e_i}{e_j}\}$ does not in general yield any topological information that was not otherwise known directly from the Chern-Gauss-Bonnet formula for $\chi(M)$ in dimension 4 (even though a normal form-argument was used in deriving the Hitchin-Thorpe inequality in \cite{hitchin}, it was not essential; see Section \ref{sec:CGB} below). Having said that, here is one case in which $\hat{S}$'s normal form basis does yield topological information:
\begin{prop}
\label{prop:pure}
If the normal form $\{\ww{e_i}{e_j}\}$ for $\hat{S}$ is ``pure", meaning that the $\ww{e_i}{e_j}$ are eigenvectors of $\hat{S}$, then $\tau(M) = 0$.
\end{prop}

\begin{proof}
As we saw above, the normal form basis $\{\ww{e_i}{e_j}\}$ yields
$$
\hat{S} = \begin{bmatrix}A&B\\B&A\end{bmatrix} \comma A = \text{diag}(a_1,a_2,a_3) \comma B = \text{diag}(b_1,b_2,b_3).
$$
In terms of these, the signature is
$$
\tau(M) = \frac{1}{3\pi^2}\int_M \sum_{i=1}^3b_i\Big(a_i+\frac{\Delta f}{4}-\frac{\text{scal}}{12}\Big)dV_{\scalebox{0.6}{$g$}}.
$$
(For the derivation of this formula, see \eqref{eqn:tau1} below.) But if the $\ww{e_i}{e_j}$ are eigenvectors of $\hat{S}$, then $b_1=b_2=b_3=0$. 
\end{proof}

(The Koiso--Cao soliton on $\mathbb{CP}^2\#(-\mathbb{CP}^2)$ has $\tau=0$.) Rather, the more interesting question is whether the normal form for $\hat{S}$ is sufficient to yield a normal form for $\co$, as was posed in \cite{cao}. Here something can be said, though once again in a special case. First, let us recall that, with respect to our normal form basis $\{\ww{e_i}{e_j}\}$,
$$
*(\ww{e_1}{e_2}) = \ww{e_3}{e_4} \commas *(\ww{e_1}{e_3}) = \ww{e_4}{e_2} \commas *(\ww{e_1}{e_4}) = \ww{e_2}{e_3}.
$$
Using \eqref{eqn:Hcoord}, and setting $e_{ij} \defeq \ww{e_i}{e_j}$, a direct calculation gives
\beqa
\hat{H}(e_{12})\!\!&=&\!\!-(f_{11}+f_{22})e_{12}-f_{23}e_{13}-f_{24}e_{14}-f_{14}e_{42}+f_{13}e_{23},\nonumber\\
\hat{H}(e_{13})\!\!&=&\!\!-(f_{11}+f_{33})e_{13}-f_{23}e_{12}+f_{34}e_{14}+f_{12}e_{42}-f_{14}e_{34},\nonumber\\
\hat{H}(e_{14})\!\!&=&\!\!-(f_{11}+f_{44})e_{14}-f_{24}e_{12}-f_{34}e_{13}+f_{12}e_{23}+f_{13}e_{34}.\nonumber
\eeqa
This allows us to determine when the basis $\{\ww{e_i}{e_j}\}$ will also be critical for the sectional curvature $\Rsec$:

\begin{prop}\label{mainthm}
Let $\{\ww{e_i}{e_j}\}$ be the normal form basis for $\hat{S}$ as in \emph{\cite{cao}}, and set $P_1 \defeq \ww{e_1}{e_2}$. Then the following are equivalent:
\begin{itemize}
\item[i.] $P_1$ is a critical point of \emph{$\Rsec$}.
\item[ii.] $f_{13}=f_{14}=f_{23}=f_{24}=0$.
\item[iii.] \emph{$\Ric(P_1,*P_1)=0$}.
\end{itemize}
The analogous statements hold cyclically for $P_2\defeq\ww{e_1}{e_3}$ and \mbox{$P_3\defeq \ww{e_1}{e_4}$.} Furthermore, if $\mathcal O_i$ denotes the orthogonal projection of $\what R(P_i)$ onto $\operatorname{span}\{P_i,*P_i\}^{\perp}$, then
\begin{align*}
|\mathcal O_1|^2&=\tfrac14\bigl(f_{13}^2+f_{14}^2+f_{23}^2+f_{24}^2\bigr),\\
|\mathcal O_2|^2&=\tfrac14\bigl(f_{12}^2+f_{14}^2+f_{23}^2+f_{34}^2\bigr),\\
|\mathcal O_3|^2&=\tfrac14\bigl(f_{12}^2+f_{13}^2+f_{24}^2+f_{34}^2\bigr).
\end{align*}
Hence
\beqa
\label{obstruction}
|\mathcal O_1|^2+|\mathcal O_2|^2+|\mathcal O_3|^2
=\frac12\sum_{i<j}f_{ij}^2
=\frac12\sum_{i<j}\emph{\Ric}(e_i,e_j)^2.
\eeqa
\end{prop}

\begin{proof}
Because $\what S(P_1)=a_1P_1+b_1(*P_1)$, the 2-plane $P_1$ is critical for $\Rsec$ if and only if $\what H(P_1)$ has no component in $\operatorname{span}\{e_{13},e_{14},e_{42},e_{23}\}$; this is ii. The equivalence with iii.~follows from \eqref{eqn:sol}, which yields $\Ric(e_i,e_j)=-f_{ij}$ for $i\neq j$. The formulas for $|\mathcal O_i|^2$ follow by inspection, and summing them shows that every off-diagonal entry appears exactly twice.
\end{proof}

\begin{cor}\label{sharedframe}
Let $P_1,P_2,P_3$ be as in Proposition \ref{mainthm}. Then the following are equivalent:
\begin{itemize}
\item[i.] $P_1,P_2,P_3$ are all critical for \emph{$\Rsec$}.
\item[ii.] \emph{$\Hess f$} is diagonal with respect to $\{\ww{e_i}{e_j}\}$.
\item[iii.] \emph{$\Ric$} is diagonal with respect to $\{\ww{e_i}{e_j}\}$.
\end{itemize}
Under these conditions, $\what R$ admits a normal form in the same frame.
\end{cor}

In other words, the off-diagonal Ricci tensor is the exact obstruction to the normal form basis $\{\ww{e_i}{e_j}\}$ of $\hat{S}$ being a normal form basis for $\co$. Interestingly, in at least one well known case there is no such obstruction:

\begin{cor}
The Koiso--Cao soliton on $\mathbb{CP}^2\#(-\mathbb{CP}^2)$ satisfies Corollary \ref{sharedframe}. Thus its curvature operator admits a normal frame in the same frame.
\end{cor}

\begin{proof}
This is proved in Theorem \ref{thm:KC} below.
\end{proof}

Short of Corollary \ref{sharedframe}, one would need some knowledge of the \emph{second} derivatives of $\Rsec$, too:

\begin{prop}
Suppose that some $P_i$ in Proposition \ref{mainthm} is critical for \emph{$\Rsec$}. Then so is $*P_i$, and $\co$ is determined by the two critical values \emph{$\Rsec(P_i),\Rsec(*P_i)$} together with the Hessian of \emph{$\Rsec$} at $P_i$.
\end{prop}

\begin{proof}
By \cite[Theorem~9]{aazami2}, if $P_i,*P_i$ are both critical 2-planes of $\Rsec$, then $\co$ can be determined from the critical values $\Rsec(P_i), \Rsec(*P_i)$, and the Hessian of $\Rsec$ at $P_i$ (this is true at any point of an oriented Riemannian 4-manifold). Thus we need only show that $*P_i$ is critical for $\Rsec$. But if $P_i$ is critical for $\Rsec$, then the same off-diagonal $\text{Hess}f$ entries that obstruct $P_i$ also obstruct $*P_i$ (e.g., the obstruction set for $P_1 = \ww{e_1}{e_2}$ and $*P_1 = \ww{e_3}{e_4}$ is the same four entries $f_{13},f_{14},f_{23},f_{24}$; similarly with the others). It follows that $*P_i$ is likewise critical for $\Rsec$.
\end{proof}

(Thus whenever one pair $P_i,*P_i$ is critical for $\Rsec$, the combination of the $\hat S$-normal form with the Hessian of $\Rsec$ at $P_i$ recovers $\co$ completely. In this sense, the $\hat{S}$-normal form plus one second-variation datum plays the role that a genuine normal form for $\co$ would play.) We close this section by placing $\hat{S}$ in the classical framework of \cite{thorpe4,zoltek}; in particular, we show that the critical pairs $\{P_i,*P_i\}$ are intrinsic Grassmannian slices of a canonical kernel of $\hat{S}$. Set $E_i\defeq \operatorname{span}\{P_i,*P_i\}$ and $G \defeq \text{Gr}_2(T_pM)$. Since $P_i,*P_i$ are critical for $\hat{S}$, we know that
$$
\hat S(P_i)=a_iP_i+b_i(*P_i) \comma \hat S(*P_i)=b_iP_i+a_i(*P_i).
$$
In particular,
$$
(\hat S-a_iI-b_i*)(E_i)=0.
$$
Hence $E_i\subseteq \ker(\hat S-a_iI-b_i*)$, and therefore
$$
\{\pm P_i,\pm *P_i\}=G\cap E_i\subseteq G\cap \ker(\hat S-a_iI-b_i*).
$$
Moreover, $G\cap E_i=\{\pm P_i,\pm (*P_i)\}$ because any unit bivector of the form $xP_i+y(*P_i)$ is decomposable only when $xy=0$. Thus, if the eigenvalues $a_i\pm b_i$ of $A\pm B$ are both simple, then $\ker(\hat S-a_iI-b_i*)=E_i$, so that
$$
G\cap \ker(\hat S-a_iI-b_i*)=\{\pm P_i,\pm *P_i\}.
$$
Without this simplicity hypothesis, the kernel may be larger, but the pair $\{P_i,*P_i\}$ always sits canonically inside it.

\section{The Spectrum of $\hat{S}=\co + \frac{1}{2}\hat{H}$}
\label{sec:proof}
In the Hodge basis, we saw in \eqref{eqn:CTsimple} that
$$
\what S=
\begin{bmatrix}
S_+&0\\0&S_-
\end{bmatrix},
\comma
S_+\defeq A+B \comma S_-\defeq A-B.
$$
Since
\beqa
\label{eqn:W}
\what W^{\pm}=S_{\pm}+\Bigl(\frac{\Delta f}{4}-\frac{\text{scal}}{12}\Bigr)I \comma
\text{scal}+\Delta f=4\lambda,
\eeqa
we also have
\begin{equation}\label{shift}
S_{\pm}=\what W^{\pm}+\Big(\frac{\text{scal}}{3}-\lambda\Big)I.
\end{equation}
This makes several familiar structures visible as spectral patterns of $S_\pm$.

\begin{prop}\label{patterns}
At any point of a gradient Ricci $4$-soliton:
\begin{itemize}
\item[i.] If $g$ is K\"ahler and the orientation is the complex orientation, then the spectrum of $S_+$ is
$$
\Bigl\{\frac{\emph{\text{scal}}}{2}-\lambda,\ \frac{\emph{\text{scal}}}{4}-\lambda,\ \frac{\emph{\text{scal}}}{4}-\lambda\Bigr\}\cdot
$$
In particular, the self-dual block has a $2+1$ eigenvalue pattern.
\item[ii.] If $W^+=0$, then $S_+=\bigl(\frac{\emph{\text{scal}}}{3}-\lambda\bigr)I$; if $W^-=0$, then $S_-=\bigl(\frac{\emph{\text{scal}}}{3}-\lambda\bigr)I$.
\end{itemize}
\end{prop}

\begin{proof}
For i., on a K\"ahler surface in the complex orientation the eigenvalues of $\what W^+$ are given by \cite{derd2}:
$$
\Big\{\frac{\text{scal}}{6},\ -\frac{\text{scal}}{12},\ -\frac{\text{scal}}{12}\Big\}\cdot
$$
Adding the scalar shift from \eqref{shift} gives the stated spectrum. Part ii.~follows immediately from \eqref{shift}.
\end{proof}

The quantity
$$
|S_+|^2-|S_-|^2=4\sum_{i=1}^3 a_ib_i
$$
is the pointwise signature density. It shows that purity (Proposition \ref{prop:pure}) is a strong sufficient condition for $\tau(M)=0$, but not a necessary one, as vanishing signature only requires the global balance
$$
\int_M (|S_+|^2-|S_-|^2)\,dV_g=0.
$$

We now proceed to our main result. Proposition \ref{patterns} shows that in the K\"ahler case the self-dual block of $\hat S$ has the same $2+1$ spectral pattern that appears in the classical $4$-dimensional K\"ahler theory. This suggests trying to import the argument in \cite{johnson} for K\"ahler curvature operators with positive-definite Ricci tensor. Of course, \cite{johnson} is a statement about $\co$, whereas the normal form of \cite{cao} is a statement about $\hat S=\co+\tfrac12\hat H$. Thus one would need the $\hat{S}$-normal form basis $\{\ww{e_i}{e_j}\}$ to be adapted to $\co$, i.e., to be in the setting of Corollary~\ref{sharedframe}. In light of the fact that the Koiso--Cao soliton on $\mathbb{CP}^2\#(-\mathbb{CP}^2)$ is K\"ahler and has positive Ricci operator, it is the natural $4$-dimensional example to test (by way of contrast, the noncompact Feldman--Ilmanen--Knopf shrinker \cite{feldman} has mixed Ricci sign near the zero section).

\begin{thm}
\label{thm:KC}
On the open dense regular set $S^3\times(\alpha,\beta)$ of the Koiso--Cao soliton, let
\beqa
\label{eqn:orozco}
e_1\defeq H \comma e_2\defeq JH=\frac{X}{f} \comma e_3\defeq \frac{Y}{h} \comma e_4\defeq Je_3=\frac{Z}{h}\cdot
\eeqa
Then $\{\ww{e_i}{e_j}\}$ is a common normal form basis for $\hat S$ and $\co$.
\end{thm}

\begin{proof}
For the properties of the $U(2)$-invariant orthonormal frame \eqref{eqn:orozco} used here, we use \cite{orozco}. For any $U(2)$-invariant function $\psi=\psi(t)$, the Hessian $\Hess\psi$ is diagonal in the orthonormal frame $\{H,\frac{X}{f},\frac{Y}{h},\frac{Z}{h}\}$. In particular $\Hess f$ is diagonal in $\{e_1,e_2,e_3,e_4\}$. By \eqref{eqn:sol}, the Ricci tensor is therefore also diagonal in this frame. Because the metric is K\"ahler, both $\Ric$ and $\Hess f$ are $J$-invariant, so their diagonal entries occur in $J$-pairs. It therefore remains to show that the above frame is a Cao--Tran normal form frame for $\hat S$. Let $\{\xi_1^{\pm},\xi_2^{\pm},\xi_3^{\pm}\}$ be the Hodge basis induced by $\{e_1,e_2,e_3,e_4\}$. Since the metric is K\"ahler and the orientation is the complex orientation, the K\"ahler form is
$$
\omega=\ww{e_1}{e_2}+\ww{e_3}{e_4}=\sqrt{2}\,\xi_1^+,
$$
while $\operatorname{span}\{\xi_2^+,\xi_3^+\}=\Lambda^{2,0}\oplus\Lambda^{0,2}$. Because the metric is K\"ahler and $\text{Hess}f$ is $J$-invariant, both $\co$ and $\hat{H}$ commute with the induced $J$-action on $\Lambda^2$. Hence $\hat{S} = \co + \frac{1}{2}\hat{H}$ is also of K\"ahler type, and so the self-dual block $S_+$ preserves the splitting
$$
\Lambda^+=\RR\omega\oplus (\Lambda^{2,0}\oplus\Lambda^{0,2}).
$$
Moreover, by Proposition~\ref{patterns}, $S_+$ has eigenvalue $\frac{\text{scal}}{2}-\lambda$ on $\RR\omega$ and the repeated eigenvalue $\frac{\text{scal}}{4}-\lambda$ on $\Lambda^{2,0}\oplus\Lambda^{0,2}$. Thus $S_+$ is diagonal with respect to $\{\xi_1^+,\xi_2^+,\xi_3^+\}$. For the anti-self-dual block $S_-$, use the $U(2)$-symmetry of the Koiso--Cao metric. On each regular orbit the isotropy is $U(1)$, which fixes $\operatorname{span}\{e_1,e_2\}$ and rotates $\operatorname{span}\{e_3,e_4\}$. Consequently,
$$
\Lambda^- = \operatorname{span}\{\xi_1^-\}\oplus \operatorname{span}\{\xi_2^-,\xi_3^-\},
$$
where the second summand is the standard real $2$-dimensional rotation representation of $U(1)$. Since $\hat S$ is $U(2)$-invariant, $S_-$ is $U(1)$-equivariant. Being also self-adjoint, it acts by a scalar on $\operatorname{span}\{\xi_2^-,\xi_3^-\}$ and hence is diagonal in the basis $\{\xi_1^-,\xi_2^-,\xi_3^-\}$. Thus both $S_+$ and $S_-$ are diagonal in the Hodge basis $\{\xi_1^{\pm},\xi_2^{\pm},\xi_3^{\pm}\}$, so this frame is a Cao--Tran normal form frame for $\hat S$. Since $\Ric$ is diagonal in the same frame, Corollary~\ref{sharedframe} applies and shows that $\co$ has a normal form in that very frame as well. Finally, although the orthonormal frame above is defined only on the regular set, the pointwise existence of a common normal form extends to the singular orbits by continuity and compactness of $SO(4)$.
\end{proof}

\begin{cor}
\label{cor:Johnson}
The curvature operator of the Koiso--Cao soliton has a normal form relative to the space of algebraic Kähler curvature operators.
\end{cor}

\begin{proof}
The Koiso--Cao soliton is K\"ahler and has positive-definite Ricci tensor \cite{orozco}. Hence \cite[Theorem~1.1]{johnson} yields a normal form for $\co$ relative to the space of algebraic K\"ahler curvature operators. (Specifically, \cite[Corollary~2.3]{johnson} yields at least two holomorphic critical planes and at least three nonholomorphic critical planes $Q_i$ for $\Rsec$ such that the bivectors $Q_i+JQ_i$ are linearly independent.) 
\end{proof}

\section{The Chern-Gauss-Bonnet formula}
\label{sec:CGB}
In this final section, we compute the Euler characteristic and signature using the $\hat{S}$-normal form; we will show that it does not in general yield any new topological information (for this and the signature formula below, see, e.g., \cite[p.~371]{besse}):
\beqa
\chi(M) \!\!&=&\!\! \frac{1}{8\pi^2}\int_M \text{tr}\Big[\Big(A+B+\frac{\Delta f}{4}I\Big)^{\!2} -\frac{1}{2}CC^t+\Big(A-B+\frac{\Delta f}{4}I\Big)^{\!2}\,\Big]dV_{\scalebox{0.6}{$g$}}\nonumber\\
&=&\!\! \frac{1}{8\pi^2}\int_M \text{tr}\Big[2\Big(A+\frac{\Delta f}{4}I\Big)^{\!2}+2B^2 -\frac{1}{2}CC^t\Big]dV_{\scalebox{0.6}{$g$}}\nonumber\\
&=&\!\! \frac{1}{8\pi^2}\int_M \bigg[2\sum_{i=1}^3\Big(a_i+\frac{\Delta f}{4}\Big)^{\!2}+2\sum_{i=1}^3b_i^2 -\frac{1}{2}\text{tr}(CC^t)\bigg]dV_{\scalebox{0.6}{$g$}}.\nonumber
\eeqa
However, $\text{tr}(CC^t)$ simplifies to
$$
\text{tr}(CC^t) = \sum_{i,j=1}^3 c_{ij}^2 = \sum_{i=1}^4\Big(f_{ii}-\frac{\Delta f}{4}\Big)^{\!2} +2\sum_{i<j}f_{ij}^2 =\Big|\underbrace{\text{Hess}f - \frac{\Delta f}{4}g}_{\text{$\defeq \mathring{\text{Hess}}f$}}\Big|^2, 
$$
so that the Euler characteristic can be written as
\beqa
\label{eqn:chi1}
\chi(M) = \frac{1}{4\pi^2}\int_M \bigg[\sum_{i=1}^3\Big(a_i+\frac{\Delta f}{4}\Big)^{\!2}+\sum_{i=1}^3b_i^2 -\frac{1}{4}|\mathring{\text{Hess}}f|^2\bigg]dV_{\scalebox{0.6}{$g$}}.
\eeqa
Next, we compute the signature, for which we will need the self-dual and anti-self-dual portions \eqref{eqn:W} of the Weyl curvature operator. The Thom-Hirzebruch formula for the signature then gives
\beqa
\tau(M) \!\!&=&\!\! \frac{1}{12\pi^2}\int_M (|\hat{W}^+|^2 - |\hat{W}^-|^2)dV_{\scalebox{0.6}{$g$}}\nonumber\\
&=&\!\! \frac{1}{12\pi^2}\int_M \bigg[\sum_{i=1}^3 \Big(a_i+b_i+\frac{\Delta f}{4}-\frac{\text{scal}}{12}\Big)^{\!2}\nonumber\\
&&\hspace{1in} - \sum_{i=1}^3 \Big(a_i-b_i+\frac{\Delta f}{4}-\frac{\text{scal}}{12}\Big)^{\!2}\bigg]dV_{\scalebox{0.6}{$g$}}\nonumber\\
&=&\!\! \frac{1}{3\pi^2}\int_M \sum_{i=1}^3b_i\Big(a_i+\frac{\Delta f}{4}-\frac{\text{scal}}{12}\Big)dV_{\scalebox{0.6}{$g$}}.\label{eqn:tau1}
\eeqa
Combining \eqref{eqn:chi1} and \eqref{eqn:tau1},
\beqa
\chi(M) - \frac{3}{2}\tau(M) \!\!&=&\!\! \frac{1}{4\pi^2}\int_M \bigg[\sum_{i=1}^3\Big(a_i+\frac{\Delta f}{4}\Big)^{\!2}+\sum_{i=1}^3b_i^2 -\frac{1}{4}|\mathring{\text{Hess}}f|^2\nonumber\\
&&\hspace{.5in}-2\sum_{i=1}^3b_i\Big(a_i+\frac{\Delta f}{4}-\frac{\text{scal}}{12}\Big)\bigg]dV_{\scalebox{0.6}{$g$}} \nonumber\\
&=&\!\! \frac{1}{4\pi^2}\int_M \sum_{i=1}^3 \Big(a_i+\frac{\Delta f}{4}-b_i\Big)^{\!2}\nonumber\\
&&\hspace{0.5in}-\frac{1}{4}|\mathring{\text{Hess}}f|^2+\frac{\text{scal}}{6}\underbrace{(b_1+b_2+b_3)}_{0}\bigg]dV_{\scalebox{0.6}{$g$}},\nonumber
\eeqa
where $b_1+b_2+b_3=R_{3412}+R_{4213}+R_{2314}=0$ via the algebraic Bianchi identity. But this is not helpful; indeed, from \eqref{eqn:W} we have that
$$
\sum_{i=1}^3 \Big(a_i+\frac{\Delta f}{4}-b_i\Big)^{\!2} = \Big|\hat{W}^-+\frac{\text{scal}}{12}I\Big|^2 = |\hat{W}^-|^2 + \frac{\text{scal}^2}{48},
$$
the latter because $\ipr{\hat{W}^-}{I} = \text{tr}(\hat{W}^-) = 0$. Finally, the trace of \eqref{eqn:sol} yields $\text{scal} + \Delta f = 4\lambda$, hence
$$
(\Ric + \text{Hess}f) - \frac{\Delta f}{4}g = \lambda g - \frac{\Delta f}{4}g \imp \mathring{\text{Hess}}f = -\mathring{\Ric},
$$
so that $|\mathring{\text{Hess}}f|^2 = |\mathring{\Ric}|^2$. Putting this together, we have
\beqa
\label{eqn:exact}
\chi(M) - \frac{3}{2}\tau(M) = \frac{1}{4\pi^2}\int_M \Big(|\hat{W}^-|^2 + \frac{\text{scal}^2}{48} - \frac{1}{4}|\mathring{\Ric}|^2\Big)dV_{\scalebox{0.6}{$g$}}.
\eeqa
Reversing orientation changes the sign of $\tau(M)$ in \eqref{eqn:exact} and swaps the self-dual and anti-self-dual Weyl blocks, so $|\hat{W}^-|^2$ is replaced by $|\hat{W}^+|^2$. This relationship, however, is already well known to hold for any compact Riemannian 4-manifold (see, e.g., \cite[Eq.~(4)]{lebrun0}). Indeed, even the fact that (non-flat) compact gradient Ricci 4-solitons must have $\chi(M) > 0$ \cite{derd} does not follow from \eqref{eqn:chi1}.

\section*{References}
\renewcommand*{\bibfont}{\footnotesize}
\printbibliography[heading=none]
\end{document}